\documentclass[reqno]{amsart}
\usepackage{amssymb}
\usepackage{amsmath}
\usepackage{amsthm}
\usepackage{pgf}
\usepackage{color}
\usepackage{graphicx}   
\usepackage{hyperref}
\newtheorem{theorem}{Theorem}[section]
\newtheorem{lemma}[theorem]{Lemma}
\newtheorem{corollary}[theorem]{Corollary}
\newtheorem{proposition}[theorem]{Proposition}

\theoremstyle{definition}
\newtheorem{definition}[theorem]{Definition}

\numberwithin{equation}{section}

\theoremstyle{definition}

\newcommand{\R}{\Bbb R}

\def\eps{\varepsilon}

\def\e{{\rm e}}

\def\weakly{\rightharpoonup}
\def\weaklystar{\stackrel{*}{\rightharpoonup}}

\newcommand{\ue}{u^\epsi}

\newcommand{\epsi}{\eps}
\renewcommand{\div}{\operatorname{div}}

\begin{document}

\title[A variational approach to Navier-Stokes]{A variational approach to Navier-Stokes}

\author{Michael Ortiz}
\address[Michael Ortiz]{California Institute of Technology, Graduate Aeronautical
  Laboratories, M.S. 105-50, Pasadena, CA 91125 USA and Institute for
  Applied Mathematics, Universit\"at Bonn, Endenicher Allee 60
D-53115 Bonn, Germany.}
\email{ortiz@aero.caltech.edu }
\urladdr{http://www.ortiz.caltech.edu}

\author{Bernd Schmidt}
\address[Bernd Schmidt]{Universit{\"a}t Augsburg, Institut f{\"u}r Mathematik, 
Universit{\"a}tsstr.\ 14, 86159 Augsburg, Germany.}
\email{bernd.schmidt@math.uni-augsburg.de}
\urladdr{http://www.math.uni-augsburg.de/ana/schmidt.html}

\author{Ulisse Stefanelli}
\address[Ulisse Stefanelli]{Faculty of Mathematics, University of Vienna, 
Oskar-Morgenstern-Platz 1, 1090 Wien, Austria  and  Istituto di Matematica
Applicata e Tecnologie Informatiche \textit{{E. Magenes}}, v. Ferrata 1, 27100
Pavia, Italy.}
\email{ulisse.stefanelli@univie.ac.at}
\urladdr{http://www.mat.univie.ac.at/$\sim$stefanelli}

\subjclass[2010]{35Q30, 76D05}
\keywords{Navier-Stokes system, variational method, weak solutions} 
 
\begin{abstract}
 
We present a variational resolution of the incompressible
Navier-Stokes system by means of stabilized
Weighted-Inertia-Dissipation-Energy (WIDE) functionals. The
minimization of these parameter-dependent functionals corresponds to an elliptic-in-time regularization of the
system. By passing to the limit in the regularization parameter along
subsequences of WIDE minimizers  one recovers a classical Leray-Hopf
weak solution.

\end{abstract}

\maketitle

\tableofcontents

%--------------------------------------------------------------------------
%--------------------------------------------------------------------------
\section{Introduction}\label{sec:Intro}
%--------------------------------------------------------------------------

This note is concerned with the 
Navier-Stokes system 
\begin{align}\label{eq:NSE-with-p}
  \partial_t u  +    u \cdot \nabla u   - \nu\Delta u 
  +  \nabla p = 0, \quad \operatorname{div} u 
  = 0, 
\end{align} 
describing the flow velocity $u:\Omega \times (0,\infty) \to \R^3$ of
an incompressible viscous fluid in a container
$\Omega \subset \R^3$. Note that we use the classical notation
$u\cdot \nabla u =   u_j \, \partial_{x_j} u$ (sum over repeated indices). We advance a
variational approach to the existence of classical {\it Leray-Hopf}
weak solutions \cite{temam84} to system
\eqref{eq:NSE-with-p} by minimizing the  
functionals $I^\eps$ on entire trajectories  
\begin{align}
  I^{\eps}(u) 
  &= \int_0^{\infty}\!\!\int_\Omega  \e^{-t/\eps}\left\{\frac{1}{2}
    |\partial_t u + u \cdot \nabla u |^2 
     + \frac{\sigma}{2} |u \cdot \nabla u |^2
     + \frac{\nu}{2\eps} |\nabla u |^2
  \right\}\, dx \, dt \label{I}
\end{align}
under the incompressibility constraint $\div u=0$ and for given initial
and (homogeneous Dirichlet) boundary 
conditions. Here, $\sigma > 0$ is a constant and $ \eps >0$ is a small parameter, eventually tending
to $0$. 

The relation between the minimization of $I^\eps$ and the Navier-Stokes
system \eqref{eq:NSE-with-p} is revealed by formally computing the
Euler-Lagrange equation for $I^\eps$ at a critical point $\ue$, namely
  \begin{align}\label{EL}
    \begin{split}
      0&=\partial_t \ue + \ue \cdot \nabla \ue -\nu  \Delta \ue + \nabla p \\
      & + \epsi \big(-\partial_t (\partial_t\ue + \ue \cdot \nabla
      \ue)- \div ((\partial_t\ue + \ue \cdot \nabla \ue) \otimes \ue)
      + (\nabla \ue)^\top
      (\partial_t\ue + \ue \cdot \nabla \ue) \big) \\
      & + \epsi \sigma \big( -\div ((\ue\cdot \nabla \ue) \otimes\ue)
      +(\nabla \ue)^\top (\ue \cdot \nabla \ue) \big).
    \end{split}
  \end{align} 
Note that the first line above is nothing but the first equation in
\eqref{eq:NSE-with-p}. The term $\nabla p $ is the Lagrangian
multiplier corresponding to the constraint $\div u=0$. The  second and
third lines in \eqref{EL} feature terms
premultiplied by the small parameter $\eps$. By taking $\epsi
\to 0$ in the Euler-Lagrange equations \eqref{EL} one formally recovers a
solution of \eqref{eq:NSE-with-p}. This paper is devoted to make this
program rigorous. Our aim is to prove that 
\vspace{3mm}

\begin{center}
  \begin{minipage}{0.8\linewidth}
    The functional $I^\eps$ admits minimizers $\ue$ for all $\epsi>0$
    (Proposition \ref{prop:min-ex}) and, up to
    subsequences, such minimizers converge to a Leray-Hopf
    solution of the Navier-Stokes system as $\eps \to 0$ (Theorem \ref{theo:main}).
  \end{minipage}
\end{center}
\vspace{3mm}

The $\eps$-dependent part of the Euler-Lagrange equations \eqref{EL}
features the term
$-\epsi \,\partial^2_t \ue$ as well. The minimization of  $I^\eps$ hence
corresponds to performing an
{\it elliptic  regularization} in time of the Navier-Stokes system
\eqref{eq:NSE-with-p}.  In particular, the minimizers $\ue$ of
$I^\eps$ are more regular in time with respect to the limiting Leray-Hopf solutions.

Elliptic regularizations of evolution problems have been
introduced by {\sc Lions} \cite{Lions:63a} and then used
by {\sc Kohn \& Nirenberg} \cite{Kohn65} and {\sc Olein\u{i}k}
\cite{Oleinik} in order to discuss regularity issues. The reader is
referred to the book by {\sc Lions \& Magenes} \cite{Lions-Magenes1} for an account of
results in the linear setting. Nonlinear systems, yet 
under stronger growth assumptions on the viscosity, including an application 
to Navier-Stokes are investigated by {\sc Lions} in \cite{Lions:63,Lions:65}. 
By way of contrast to our set-up, for the elliptic regularizations in these 
contributions no variational structure is available. 

The novelty of our contribution is that of directly moving from a global-in-time
variational perspective. The $\eps$-dependent functionals $I^\eps$ correspond to
the {\it Weighted
  Inertia-Dissipation-Energy} (WIDE) functionals for viscous-fluid flows. They
are obtained as the weighted sum of inertia,
dissipation, and energy of the fluid. We present an account of this derivation
in Section
\ref{sec:WED} below. Let us however stress that, in addition to the
above-mentioned classical terms, we include in the analysis a $\sigma$
stabilization term. This is instrumental in proving a-priori estimates
and, as it is apparent from inspecting the Euler-Lagrange equations \eqref{EL},
has no influence on the limit system. 
We note that stabilization is a standard tool in numerical methods for compressible and incompressible flows (cf., e.g., \cite{Hughes} for a review). The specific regularization employed in this work is referred to as `streamline-upwind/Petrov-Galerkin' (SUPG) regularization in the numerical literature and was first introduced in \cite{Brooks}.

The study of
WIDE functionals started from the case of gradient flows (no inertia,
quadratic dissipation) and has to be traced back at least to 
{\sc Ilmanen} \cite{Ilmanen94}, who used a global-in-time variational
method to tackle 
existence and partial
regularity of the  Brakke mean-curvature flow of
varifolds. An application to existence of periodic
solutions for gradient flows is given by {\sc
  Hirano} in \cite{Hirano94}. The variational nature of elliptic
  regularization is at the core of \cite[Problem~3, p.~487]{Evans98}
  of the classical textbook by {\sc Evans}. Two examples of
relaxation related with micro-structure evolution have been provided
 in  \cite{Conti-Ortiz08} and the case of mean-curvature evolution
of Cartesian surfaces is in \cite{spst}. The analysis of the WIDE
approach for abstract gradient flows for $\lambda$-convex and nonconvex energies is
in \cite{ms3,akst5} in the Hilbertian case and in \cite{cras,metric_wed} in
the metric case.  {\sc Melchionna} \cite{Melchionna0} extended the
theory to classes of nonpotential perturbations and {\sc B\"ogelein, Duzaar, \&
Marcellini}~\cite{Boegelein-et-al14} recently used this variational approach to
prove the existence of variational
solutions to the equation 
$
u_t -\nabla \cdot f(x,u,\nabla u) + \partial_u f(x,u,
\nabla u)=0
$
where the field $f$ is convex in
$(u, \nabla u)$.

Doubly nonlinear parabolic evolution equations can be tackled by the
WIDE variational formalism as
well. The first result in this direction is in  
\cite{Mielke-Ortiz06}, where the case of rate-independent processes is
addressed. The corresponding time
discretization has been presented in the subsequent \cite{ms2} and  an
application to crack-front propagation in brittle materials  is in
\cite{Larsen-et-al09}. 
The rate-dependent case  has been analyzed in 
\cite{akmest,akst,akst2,akst4}. See also {\sc Liero \& Melchionna} \cite{Liero-Melchionna} for
a stability result via $\Gamma$-convergence \cite{DalMaso93} and
{\sc Melchionna} \cite{Melchionna} for an application to the study of
symmetries of solutions.

In the dynamic case, {\sc De Giorgi} conjectured in
\cite{DeGiorgi96} that the WIDE functional procedure could be
implemented in the setting of semilinear waves.  This 
  has been ascertained in \cite{dg} (for the finite-time case)
and in {\sc Serra \& Tilli} \cite{Serra-Tilli10} (for the
infinite-time case).   The possibility of following this same
  variational approach in other hyperbolic situations has also been
  pointed out in \cite{DeGiorgi96}. Indeed, extensions to mixed
hyperbolic-parabolic semilinear equations \cite{dg3}, to different
classes of nonlinear energies \cite{dg2,Serra-Tilli14}, and to
nonhomogeneous equations \cite{TTilli17} are also available.  
  Here, we
  further develop De Giorgi's approach by showing its
  applicability to fluids. When formulated in terms of the fluid
  motion $\varphi:\Omega \times [0,\infty) \to \Omega$, fluid dynamics
  is hyperbolic. The corresponding 
 WIDE functional \eqref{eq:WIDE} is derived in Section
  \ref{sec:WED}. Such functional is then recast in form of $I^\epsi$ by changing
  variables from the Lagrangian motion $\varphi$ to the Eulerian
velocity field $u = \dot \varphi \circ \varphi^{-1}$.

The literature on the Navier-Stokes system is huge and we shall not
attempt to review it here.   The reader is referred to
  \cite{Robinson,Seregin,temam84} for a collection of results.
On the other hand, global-in-time variational
approaches to  
\eqref{eq:NSE-with-p} are just a few. A complete variational
resolution has been provided by {\sc Ghoussoub}  within the theory of
self-dual Lagrangians \cite{Ghoussoub08}. In
particular, in \cite{Ghoussoub07,Ghoussoub09} the Navier-Stokes
system is reformulated in terms of a so-called null-minimization
principle, namely the attainment the value $0$ of
a specific nonnegative functional inspired by Fenchel's duality. Such
attainment is then ascertained in \cite{Ghoussoub07,Ghoussoub09},
giving rise to a complete existence theory for periodic-in-time solutions.

In \cite{Pedregal12} {\sc Pedregal} characterizes weak solutions as those $u$
such that  $E(v)=0$ for a given error functional $E$ computed on the
function $v=Tu$ where the mapping $T$ is defined by minimizing a second functional
parametrized by $u$. This approach provides an
existence theory in 2D only.

Solutions of the Navier-Stokes
system are related by {\sc Gomes}
\cite{Gomes05,Gomes07} to critical points of a stochastic control problem on the group
of area-preserving diffeomorphisms, see also \cite{Arnaudon12,Yasue83}
for some related discussion. Note however that the analysis in
\cite{Gomes05,Gomes07} assumes sufficient smoothness. In particular, it does
not directly originate an existence theory.

Moving from the transport-diffusion scheme by {\sc Pironneau} \cite{Pironneau},
{\sc Gigli \& Mosconi} \cite{Gigli-Mosconi12} present a variational time
discretization based on minimizing
movements and prove convergence as the fineness of the
time-partition goes to $0$.
A saddle-point formulation in space-time is investigated from the
numerical viewpoint by {\sc Schwab \& Stevenson}
\cite{Schwab17}.

%--------------------------------------------------------------------------
%--------------------------------------------------------------------------
\section{Weighted  Inertia-Dissipation-Energy  functionals for viscous fluids}\label{sec:WED}
%--------------------------------------------------------------------------

 We devote this section to discuss the Weighted
Inertia-Dissipation-Energy (WIDE) variational approach to general rate problems as well as its application to the motion of
incompressible viscous fluids.
This brings to a justification of the particular form of the
functional $I^\eps$. 

Many physical systems are described by models in rate form: The
state of the system is described by a trajectory $\varphi:   [0,\infty) \to
H$ in the Hilbert space $H$ solving  the nonlinear differential inclusion
\begin{subequations}\label{eq:Intro:Classical}
\begin{align}
    & 0 \in
    \rho \ddot{\varphi}(t)
    + \partial_{\dot{\varphi}} \Psi(\varphi(t), \dot \varphi (t))
    +
    \partial_\varphi E(t,\varphi (t)),
    \label{eq:Intro:Classical1} \\
    & \varphi (0) = \varphi_0, \quad \dot{\varphi}(0) = \varphi^1.
    \label{eq:Intro:Classical2}
\end{align}
\end{subequations}
 Here, dots represent time derivation, $\rho > 0$, $\Psi : { H \times H}
\to  [0,\infty]$ is  the  dissipation potential which is convex in the
second argument,  and  $E : {  [0, \infty) \times } H \to { \R \cup
  \{\infty\}}$ is  the time-dependent  energy function. 
Relation \eqref{eq:Intro:Classical1} results from the balance between
the system of
inertial forces $\rho \ddot{\varphi}$, the system of dissipative forces
$\partial_{\dot{\varphi}} \Psi$, and that of conservative forces
$\partial_\varphi E$, where $\partial$ stands for the partial
(sub)differential. The 
trajectory $t \mapsto \varphi(t)$  of the system is the result of this balance and of
the initial conditions \eqref{eq:Intro:Classical2}. 

A time discretization of Problem \eqref{eq:Intro:Classical} is
given by  
\begin{align}\label{eq:Intro:InfSeq}
    \inf_{\varphi_{n+1} \in H} F_{n+1}(\varphi_{n+1} ; \varphi_n),
    \quad
    n = {1, 2, \dots, } 
\end{align}
for $\varphi_n = \varphi (t_n)$, where
\begin{align*}
\begin{split}
    F_{n+1}(\varphi_{n+1} ; \varphi_n)
    & =
   \tau \frac{\rho}{2}
      \left\| 
    \frac{\varphi_{n+1} - 2 \varphi_n + \varphi_{n-1}}{\tau^2}
     \right\|^2_H  
    \\ & \quad + 
    \Psi\left(\varphi_n, \frac{\varphi_{n+1} - \varphi_n}{\tau}\right)
    -
    \Psi\left(\varphi_{n-1}, \frac{\varphi_n - \varphi_{n-1}}{\tau}\right)
    \\ & \quad
    +
    \frac{
    E(t_{n+1}, \varphi_{n+1}) - 2 E(t_n, \varphi_n) + E(t_{n-1}, \varphi_{n-1})
    }{
   \tau
    }
\end{split}
\end{align*}
 and $\tau >0$ is the time step. 
We verify that, indeed, the Euler-Lagrange equations of this problem
are
\begin{align*}
    0 \in \rho \, \frac{\varphi_{n+1} - 2 \varphi_n + \varphi_{n-1}}{\tau^2}
    +
    {\partial_{\dot{\varphi}} } \Psi \left(\varphi_n, \frac{\varphi_{n+1} - \varphi_n}{\tau}\right)
    + {\partial_{\varphi}} E(t_{n+1}, \varphi_{n+1})
\end{align*}
which may be regarded as a central-difference/backward-Euler time
discretization of \eqref{eq:Intro:Classical1}. The causal nature of Problem \eqref{eq:Intro:Classical} is
reflected in the fact that the minimizations 
\eqref{eq:Intro:InfSeq} are solved \emph{sequentially}: problem $n=1$ is solved first with initial conditions $\varphi_0$ and
 $\varphi_1=\varphi_0 + \tau \,\varphi^1$  in order to
compute $\varphi_2$; subsequently, problem $n=2$ is solved to compute
$\varphi_3$, taking the solution $\varphi_2$ of the preceding problem
and $\varphi_1$ as initial conditions; and so on.   Alternatively, by
following   \cite{Mielke-Ortiz06} we proceed to formulate a single minimum problem for the entire trajectory $\varphi = {\{\varphi_2, \varphi_3, \dots\}}$ by \emph{stringing together} all the incremental problems (\ref{eq:Intro:InfSeq}) with Pareto weights $\e^{-t_n/\eps}$, where $\eps > 0$ is a small parameter. The resulting functional is
\begin{align}\label{eq:F1}
\begin{split}
    F^{\eps}( \{\varphi_2,\varphi_3,\dots\} ; \tau ) 
    & = 
    \sum_{n=1}^{ \infty} \e^{-t_{n+1}/\eps}
    \Bigg\{
    \frac{\rho}{2}
     \left\| 
    \frac{\varphi_{n+1} - 2 \varphi_n + \varphi_{n-1}}{\tau^2}
    \right\|^2_H 
    \\ & \quad +
    \frac{1}{\tau}
    \left[
    \Psi\left(\varphi_n, \frac{\varphi_{n+1} - \varphi_n}{\tau}\right)
    -
    \Psi\left(\varphi_{n-1}, \frac{\varphi_n - \varphi_{n-1}}{\tau}\right)
    \right]
    \\ & \quad +
    \frac{E(t_{n+1}, \varphi_{n+1}) - 2 E(t_n, \varphi_n) + E(t_{n-1}, \varphi_{n-1})}
    {\tau^2}
    \Bigg\} \, \tau. 
\end{split}
\end{align}
In the \emph{causal limit} of $\eps\to 0$, the exponential weights accord disproportionately larger importance to the first incremental problem relative to the second; to the second incremental problem relative to the third, and so on, as required by \emph{causality}. The functional \eqref{eq:F1} may formally be regarded as a time discretization of the continuous-time functional
\begin{align*}
    F^{\eps}(\varphi)
    =
    { \int_0^{\infty}} \e^{-t/\eps}
    \left\{
    \frac{\rho}{2} \| \ddot{\varphi}{ (t)}\|^2_H
    + { 
    \frac{d}{dt} \Psi \big(\varphi(t), \dot{\varphi}(t)\big)
    +
    \frac{d^2}{dt^2} E \big(t,  \varphi  (t)\big) }
    \right\} \, dt
\end{align*}
 and integration by parts gives  
\begin{align*} 
\begin{split}
    F^{\eps}(\varphi)
    & =
    { \int_0^{\infty}} \e^{-t/\eps}
    \left\{
    \frac{\rho}{2} \| \ddot{\varphi}(t) \|^2_H
    +
    \frac{1}{\eps}\Psi\big(\varphi(t), \dot{\varphi}(t)\big)
    +
    \frac{1}{\eps^2} E\big(t,\varphi(t)\big)
    \right\} \, dt
    \\ & \quad +
    \Big[\e^{-t/\eps} \Big(\Psi \big(\varphi(t),
\dot{\varphi}(t)\big) + \dot{E}\big(t,\varphi(t)\big)\Big)\Big]_0^{{ \infty}}
    +
    \frac{1}{\eps}
    \Big[\e^{-t/\eps} E\big(t,\varphi(t)\big) \Big]_0^{{ \infty}}. 
\end{split}
\end{align*}
 Assume  that the terms  $\e^{-t/\eps} (\Psi(\varphi(t),
\dot{\varphi}(t))+
\dot{E} (t,\varphi(t)))$ and $\e^{-t/\eps} E (t,\varphi(t))$  
vanish at $t = \infty$  and drop
them  at $t = 0$,  for these just depend on the fixed initial
conditions and do not affect the minimization.   This leads to the
 general WIDE  functional 
\begin{align} 
     F^{\eps} (\varphi)
    & =
    \int_0^{\infty} \e^{-t/\eps}
    \left\{
    \frac{\rho}{2} \| \ddot{\varphi}(t) \|^2_H
    +
    \frac{1}{\eps}\Psi\big(\varphi(t), \dot{\varphi}(t)\big)
    +
    \frac{1}{\eps^2} E\big(t,\varphi(t)\big)
    \right\} \, dt.\label{eq:dg}
\end{align} 
Assuming sufficient differentiability, the Euler-Lagrange equation of
$ F^\eps$ is 
\begin{align*}
    \frac{d^2}{dt^2}
    \left(\e^{-t/\eps} \rho \ddot{\varphi} \right)
    +
    \e^{-t/\eps} \frac{1}{\eps} \partial_{\varphi} \Psi(\varphi,\dot{\varphi})
    -
    \frac{d}{dt}\left( \e^{-t/\eps}
    \frac{1}{\eps}\partial_{\dot{\varphi}}\Psi(\varphi,\dot{\varphi}) \right)
    +
    \e^{-t/\eps} \frac{1}{\eps^2} \partial_{\varphi} E(t,\varphi)
    \ni 0
\end{align*}
which, upon simplification, reduces to
\begin{align*} 
\begin{split}
    0& \in \rho \ddot{\varphi}
    +
    \partial_{\dot{\varphi}} \Psi(\varphi,\dot{\varphi})
    +
    \partial_{\varphi} E(t,\varphi)
    \\ & \quad -
    \eps \big[
    2 \rho \dddot{\varphi}
    -
    \partial_{\varphi} \Psi(\varphi,\dot{\varphi})
    +
    \partial_{\varphi} \partial_{\dot{\varphi}} \Psi(\varphi, \dot{\varphi}) \dot{\varphi}
    +
    \partial^2_{\dot{\varphi}} \Psi(\varphi, \dot{\varphi}) \ddot{\varphi} \big]
    +
    \eps^2 \rho \ddddot{\varphi}
    .
\end{split}
\end{align*}
 The  original problem \eqref{eq:Intro:Classical1} is formally
recovered in the limit of $\eps \to 0$  as the terms in the second
line drop. Note that the term $\eps^2 \rho \ddddot{\varphi}$ qualifies
the latter as an elliptic regularization of relation
\eqref{eq:Intro:Classical1}. 

  As mentioned in the Introduction, the above described
  variational procedure has been conjectured by {\sc De~Giorgi}
  \cite{DeGiorgi96} to be amenable for general functionals of the
  calculus of variations of the form \eqref{eq:dg} (actually $\Psi=0$
  is assumed  in
  \cite{DeGiorgi96}). Our results confirm this possibility in the
  specific case of incompressible fluids. Let  $\varphi :
\Omega \times  [0,\infty) \to \Omega \subset \mathbb{R}^3$ be
the motion of a Newtonian viscous
fluid in a container $\Omega$. For simplicity, we suppose that the
fluid is free of body forces  (these can easily be added)  and
the deformation mapping $\varphi$ vanishes at the boundary
$\partial\Omega$ of the container  (no slip boundary
conditions). Let $\rho : \Omega \to (0,\infty)$ denote the density
of the fluid at time $t=0$ and assume   Newtonian viscosity, i.~e.,
 that  the viscous part $\sigma^v$ of the Cauchy stress tensor, and the rate of deformation tensor
\begin{align*}
    d = {\rm sym}(\dot{F}F^{-1})  = \frac12 \big((\dot{F}F^{-1})+(\dot{F}F^{-1})^\top \big)
\end{align*}
obey the relation
\begin{align} \label{eq:Newtonian}
    \sigma^v =
    \lambda \,\mathrm{tr}(d) I
    + 2 \mu \, {\rm dev}(d) 
\end{align}
where $F = \nabla \varphi$ is the deformation gradient, ${\rm dev}(d)
= d - \mathrm{tr}(d) I/3$ is the deviatoric part of $d$, $\lambda$
and $\mu$ are  the bulk and the shear  viscosity parameters,   and $I$ is the
identity $2$-tensor.  The viscosity law \eqref{eq:Newtonian} may be expressed explicitly as a function of $\dot{F}$ and $F$ in the form
\begin{align*} 
    \sigma^v(F, \dot{F}) =
    \lambda \mathrm{tr}(\dot{F} F^{-1}) I
    + \mu \,  {\rm dev}  (
    \dot{F} F^{-1} + F^{-\top} \dot{F}^\top ). 
\end{align*}
The corresponding viscous part of the first  Piola-Kirchhoff  stress tensor is
\begin{align*} 
    P^v(F, \dot{F}) = J {\sigma}^vF^{-\top},
\end{align*}
where $J = \det(F)$. The total stress is then,
\begin{align*}
    P(F, \dot{F}) = - \pi I + P^v(F, \dot{F})
\end{align*}
where $\pi$ is the hydrostatic pressure of the fluid.  A simple calculation reveals that the Newtonian viscosity law
possesses the potential structure
\begin{align*} 
    P^v =
    \partial_{\dot{F}} \psi(F, \dot{F}), 
\end{align*}
where the viscous potential per unit undeformed volume is
\begin{align*} 
    \psi(F, \dot{F}) = J \, \left\{
    \frac{\lambda}{2} \mathrm{tr}(d)^2
    + \mu \, | {\rm dev}(d) |^2  \  \right\}
\end{align*} 
For later reference we also introduce the viscous potential per unit deformed volume as
\begin{align*}
    \psi(d) =
    J^{-1} \psi(F, \dot{F}) =
    \frac{\lambda}{2} \mathrm{tr}(d)^2
    + \mu \, | {\rm dev}(d) |^2 , 
\end{align*}
which has the property that
\begin{align*}
    \sigma^v = \partial \psi(d).
\end{align*}
The viscous potential  reads hence 
\begin{align*}
    \Psi(\varphi, \dot{\varphi})
    =
    \int_\Omega
    \psi(\nabla\varphi, \nabla\dot{\varphi}) \, dx.
\end{align*}
 On the other hand,  the energy is  given by 
\begin{align*}
    E(\varphi)
    =
    \int_\Omega
    e( J) \, dx
\end{align*}
where 
$e(J)$ is the internal energy density as a function of specific volume. 

 Let us now restrict our attention to incompressible fluids. In
this case,  
 the internal energy reduces to
\begin{align*}
     e(J) 
    =
    \left\{
    \begin{array}{ll}
    0 &
    \text{if }  J  = 1 \text{ a.e. in } \Omega, \\
    \infty & \text{otherwise}
    \end{array}
    \right.
\end{align*}
and we get $\mathrm{tr}(d) = 0$ in \eqref{eq:Newtonian}, and let 
$\rho = \rho_0$ constant and   $\nu =\mu/\rho_0$ be the  kinematic 
viscosity coefficient. By using the Eulerian velocity field
$v = \dot{\varphi}\circ\varphi^{-1}$ we can transform the
representation of the motion from Lagrangian to Eulerian. The
incompressibility constraint $J=1$ along the flow corresponds via
Jacobi identity to $\div v =0 $ and we can rewrite  the viscous
dissipation potential as 
\begin{align*}
    \int_\Omega \psi(d) \, dx  =
    \mu  \int_\Omega | {\rm dev}({\rm sym}\nabla v) |^2\, dx
    = \mu  \int_\Omega | {\rm sym} \nabla v |^2\, dx. 
\end{align*}
Hence, the WIDE functional 
\begin{align}
  F^\eps(\varphi)&=\int_0^{\infty} \e^{-t/\eps}
     \int_\Omega\left\{
     \frac{\rho}{2} |\ddot{\varphi}|^2 
     +
     \frac{1}{\eps} \psi(\nabla\varphi, \nabla\dot{\varphi})   
     +
     \frac{1}{\eps^2} e(J)
     \right\} \, dx \, dt  \label{eq:WIDE}
\end{align} 
can be rewritten as 
\begin{align*}
& \frac{1}{\rho_0} F^\eps(\varphi) \equiv I^\eps_0(v):=
\left\{
  \begin{array}{ll}
&\displaystyle\int_0^{\infty} \e^{-t/\eps}
     \int_\Omega\left\{\frac{1}{2} | \partial_t v + v \cdot \nabla v |^2
    +
    \frac{\nu}{\eps} |{\rm sym} \nabla v|^2
    \right\}   \, dx \, dt \\[3mm]
&\hspace{24.8mm}
    \text{if} \ \   \div v =0  \ \ \text{a.e. in} \ \  \Omega\times (0,\infty), \\[3mm]
    &\infty  \qquad \qquad \qquad  \text{otherwise}.
  \end{array}
\right. 
\end{align*}

Assuming sufficient differentiability, the  corresponding 
Euler-Lagrange equations  read as 
\begin{align*}
\begin{split}
    & -
    \partial_t \left( \e^{-t/\eps}
    \rho_0 ( \partial_t v_i + v_j \, \partial_{x_j} v_i ) \right)
    +
    \e^{-t/\eps}
    \rho_0 ( \partial_t v_k + v_j \, \partial_{x_j} v_k) \partial_{x_i} v_k
    \\ & \quad\quad -
    \partial_{x_k} \left( \e^{-t/\eps}
    \rho_0 ( \partial_t v_i + v_j \, \partial_{x_j} v_i ) v_k \right)
    -
    \frac{1}{\eps}
    \e^{-t/\eps} (\mu \Delta v_i - \partial_{x_i} \pi ) = 0
\end{split}
\end{align*} 
(repeated indices to be summed) for $i = 1, \ldots, n$, which, upon simplification, reduces to
\begin{align*}
\begin{split}
    0=&
     ( \partial_t v_i + v_j \, \partial_{x_j} v_i ) - \nu 
     \Delta v_i + \partial_{x_i} p
    \\ & +
    \eps \left[
    - \partial_t \left(
     ( \partial_t v_i + v_j \, \partial_{x_j} v_i ) \right)
    +
     ( \partial_t v_k + v_j \, \partial_{x_j} v_k ) \, \partial_{x_i} v_k
    -
    \partial_{x_k} \left(
     ( \partial_t v_i + v_j \, \partial_{x_j} v_i ) v_k \right)
    \right] 
\end{split}
\end{align*}
where $p = \pi/\rho_0$ is the pressure per unit mass. 
 The   incompressible Navier-Stokes  system
\eqref{eq:NSE-with-p} can hence be formally recovered in the causal limit of $\eps\to 0$.

 The analysis of the functional $I^\eps_0$ is quite challenging, for
the nonlinear term $\partial_t v + v \cdot \nabla v $ couples the time
derivative $\partial_t v$ and the nonlinear convection term $v \cdot
\nabla v$ and a separate control of these two terms seems not 
available. We hence resort in {\it stabilizing} $I^\eps_0$
by augmenting it with the additional term 
\begin{equation}
\int_\Omega \frac{\sigma}{2}|v \cdot \nabla v|^2 \,
dx.\label{eq:stab}
\end{equation}
Up to the addition of such stabilization term, the functional $I^\eps$ from \eqref{I} and the WIDE 
functional $I^\epsi_0$ coincide (note that $2\|{\rm sym} \nabla v\|^2=\|\nabla v\|^2$ on
divergence-free fields vanishing at the boundary). The inclusion of the $\sigma$ stabilizing term does not affect the limit $\eps
\to 0$ of the Euler-Lagrange equations. Indeed, any additional term
 of order one  in $\eps$   which does not contain a time
derivative will not appear to leading order in the Euler-Lagrange
equations. In particular, such a term  will vanish  in the
causal limit $\eps \to 0$. Yet, as we will see,  such
stabilization allows
 for significantly improved a-priori estimates. 

 The specific
form \eqref{eq:stab} for a stabilization term is not the only possibility. As will become clear from our analysis
(cf.\  Proposition  \ref{lemma:a-priori-bds}), other 
choices  would lead to the same a-priori estimates and, eventually, to
the same convergence result.

As mentioned in the Introduction, the minimization of $I^\eps$
  delivers an elliptic-in-time approximation
  of the incompressible Navier-Stokes system. In particular, the minimizers $I^\eps$ are
  more regular in time with respect to  
  Leray-Hopf solutions. As it will be clear from the proof of our main
result, such additional time regularity is lost as $\eps \to 0$. This
singular-perturbation procedure has indeed been pioneered by
\cite{Lions:63,Lions:65}, who proved existence of Leray-Hopf solutions
by augmenting the incompressible Navier-Stokes system by
$-\eps \partial_t^2 v$ and taking the limit as $\epsi\to 0$. In
constrast to the WIDE approach,  Lions' method is nonvariational, for
the $\epsi$-approximating system is not an
Euler-Lagrange equation of a functional.

Before moving on, let us mention that weak solutions of the
incompressible Navier-Stokes system can be tackled by a variety of
alternative methods.
The classical approach to existence is via Faedo-Galerkin approximations
\cite{Tartar,temam84}. More specifically, finite elements \cite{Girault} as well as
spectral \cite{Canuto,Guo}, stabilized \cite{Brooks,Hughes},
isogeometric \cite{Evans}, and collocation \cite{Bassi} approximations have been deeply studied. 
Time discretizations \cite{Lubich,Mueller}, including {\it parareal} \cite{Fischer}
and transport-diffusion schemes
\cite{Gigli-Mosconi12,Pironneau} have been set forth, giving rise
to space-time schemes
\cite{Schwab17,temam84}. Finite differences \cite{Strikwerda} and
finite volumes \cite{Jones} schemes have
been considered as well. Another option are
regularization and iterative methods, possibly in combination with
suitable reformulations of the problem \cite{Berselli,Chorin,Lin,Rannacher}.

%--------------------------------------------------------------------------
%--------------------------------------------------------------------------
\section{Statement of the main result}
%--------------------------------------------------------------------------
 
Let us start by setting up the functional frame.  
Let $\Omega \subset \R^3$ be an open bounded set  with Lipschitz
boundary $\partial \Omega$. 
In the following, we indicate with $\| \cdot \|_{E}$ the norm of the 
Banach space $E$ and by $\| \cdot \|$ the norm of a square
integrable function, regardless of the number of its components.
 Let  $\mathcal{V} := \{v\in
C^{\infty}_{\rm c}(\Omega;  \R^3 ) :  \div  v = 0\}$
 be given and  
introduce the Hilbert spaces 
\begin{align*}
  V_s &:= \mbox{the closure of } \mathcal{V} \mbox{ in } H^s(\Omega;  \R^3 ) \cap H^1_0(\Omega;  \R^3 ),\\
  H &:= \mbox{the closure of } \mathcal{V} \mbox{ in } L^2(\Omega;  \R^3 ), 
\end{align*}
for $s \ge 1$. If $s = 1$ we simply write $V = V_s$. Identifying $H$ with its dual $H'$ we have the dense and continuous inclusions
\begin{align*}
  V_s \hookrightarrow V \hookrightarrow H \equiv H' \hookrightarrow V' \hookrightarrow V_s'.
\end{align*}
Suppressing $s$, the dual pairing between $V_s'$ and $V_s$ is denoted
by $\langle\cdot,\cdot\rangle$: $V_s'\times V_s\to \R$. We denote by $A : V
\to V' \, (\hookrightarrow V_s'$, $s \ge 1$) the {\it Stokes} operator 
given by 
\begin{align}\label{eq:A-def}
  \langle A v, \psi \rangle 
  & =  \int_{\Omega} \nabla v : \nabla \psi  \, dx 
   \qquad\forall \, v, \psi \in V,  
\end{align}
which satisfies 
\begin{align}\label{eq:Au-bd} 
  \| A v \|_{V'} 
  &\le C \| \nabla v \|
   \qquad\forall \, v \in V.
\end{align}
Let us now define the quadratic functional $B:V \to V'$ as
\begin{align}\label{eq:B-def}
  \langle B(v), \psi \rangle 
   := \int_{\Omega} [v \cdot \nabla v] \cdot \psi \, dx = -
   \int_\Omega (v \otimes v) : \nabla \psi\, dx
  \qquad \forall \, v, \, \psi \in V. 
\end{align}
Note that $ [v \cdot \nabla v] \cdot \psi \in L^{6/5}(\Omega)$ for all
$v, \, \psi \in V$ and that we have 
\begin{align}\label{eq:B-bd}
  \| B(v) \|_{V'} 
  \le C \| |v|^2 \| \leq C \| v \|_V^2
  \qquad \forall \, v \in V  
\end{align}
and, for all $s > 3/2$ 
\begin{equation}
  \label{piu}
  \| B(v) \|_{V'_s} 
  \le C \| v \cdot \nabla v \|_{L^1(\Omega;\R^3)} \leq C \| v \|_V^2
  \qquad \forall \, v \in V_s
\end{equation}
where we have used that $V_s  \hookrightarrow
L^\infty(\Omega;\R^3)$. 

For a given initial condition $u_0 \in H$ we choose a sequence $u_0^{\eps} \in V$ 
(by using, e.g., that $\mathcal{V}$ is dense in $H$) with
\begin{align}\label{eq:u-0-bd} 
  u_0^{\eps} \to u_0 \mbox{ in } H 
  \qquad \mbox{ and } \qquad 
  \| \nabla u_0^{\eps} \|^2
    + \eps \| u_0^{\eps} \cdot \nabla u_0^{\eps} \|^2\le C_0 \eps^{-1} 
\end{align} for  some  constant $C_0 > 0$ and  define the
admissible set of trajectories as  
\begin{align*}
  U^{\eps} &:= \{u\in L^2_{\rm loc}(0,\infty;V) : \partial_t u + u
  \cdot \nabla u, \ \sigma  \,u
  \cdot \nabla u \in L^2_{\rm loc}(0,\infty; L^2(\Omega;\R^3)),  u(0) = u_0^{\eps}\}.
\end{align*}
Note that, for all $u \in U^\epsi$, from $u\in L^2_{\rm
  loc}(0,\infty;V)$ one has that $u \cdot \nabla u \in L^1_{\rm
  loc}(0,\infty;L^{3/2}(\Omega;\R^3))$. Hence $\partial_t u  \in L^1_{\rm
  loc}(0,\infty;L^{3/2}(\Omega;\R^3))$ as well,
 $t \mapsto u(t)$ is
continuous in $L^{3/2}(\Omega;\R^3)$, and the initial
condition $u(0) = u_0^{\eps}$ in the definition of
$U^{\eps}$ makes sense.

\begin{definition}[Leray-Hopf solutions] We say that $u \in
  L^2(0,\infty;V)$ is a \emph{Leray-Hopf solution} of the
  Navier-Stokes system \eqref{eq:NSE-with-p} if $\partial_t u \in
  L^{1}(0,\infty;V')$ and
  \begin{align}
    & \partial_t u +  B(u) +\nu Au =0 \quad \text{in} \ \
    V', \ \ \text{a.e. in} \ (0,T),\label{eq:NS}\\
    &u(0)=u_0  \quad \text{in} \ \  H.\label{eq:iniz}
  \end{align}   
\end{definition}
Existence of Leray-Hopf solutions is rather classical
\cite[Thm.~III.3.1]{temam84}. These are actually weakly continuous
from $[0,T]$ to $H$ \cite[Thm.~III.3.1]{temam84} so that the initial
condition \eqref{eq:iniz} makes sense. Moreover, the energy inequality \eqref{eq:ei} can be guaranteed to hold true. In three dimensions
they belong to $L^{r}_{\rm loc}(0,\infty;L^s(\Omega;\R^3))$ if $2/r + 3/s = 3/2$ and $2 \le s \le 6$ \cite[Lemma 3.5]{Robinson}, with $\partial_t u \in
 L^{4/3}_{\rm loc}(0,\infty;V') \cap L^2(0, \infty; V_{3/2}')$, $B(u) \in L^{4/3}_{\rm
   loc}(0,\infty;V')$ \cite[Thm.~III.3.3]{temam84}, and $u \in L^r_{\rm loc}(\delta, \infty; W^{2,s}(\Omega;\R^3)) \cap W^{1,r}_{\rm loc}(\delta, \infty; L^s(\Omega;\R^3))$ for any $\delta > 0$ if $2/r + 3/s = 4$ \cite[Thm.~V.6.11]{Seregin}. 
Note that the pressure $p$ plays
the role of the Lagrange multiplier corresponding to the
incompressibility constraint $u(t) \in V$ and as such does not show up
in \eqref{eq:NS}. Equivalently, one could reformulate the problem in
the pair $(u,p)$, see \cite[Rem.~I.1.4 and Eq.~(III.3.129)]{temam84}.  Leray-Hopf solutions
 are unique in two dimensions \cite[Thm.~III.3.2]{temam84}. In three
 dimensions, if $u \in
 L^{r}_{\rm loc}(0,\infty;L^s(\Omega;\R^3))$ with $2/r + 3/s = 1$ and $s > 3$ and $u_0 \in V$, then $u$ is unique among Leray-Hopf solutions satisfying the energy inequality \cite[Thm.~8.19]{Robinson}. 

For $\sigma\geq 0$ and  $ \nu>0$ 
we consider the  WIDE functionals $I^{\eps}:U^{\eps} \to
[0,\infty)$  
\begin{align*} 
  I^{\eps}(u) 
  &= \int_0^{\infty} \e^{-t/\eps}\left\{ \frac{1}{2}  \| \partial_t u + u \cdot \nabla u \|^2
     + \frac{\sigma}{2} \| u \cdot \nabla u \|^2
     + \frac{\nu}{2\eps} \| \nabla u \|^2
  \right\} \, dt,  
\end{align*}
taking values in $[0,\infty]$. (Note that any divergence-free Sobolev
functions on which $I^{\eps}$ is finite  necessarily belongs to $U^{\eps}$.) 

In Proposition \ref{prop:min-ex} below we will see that $I^{\eps}$
indeed admits a minimizer in $U^{\eps}$ for all $\sigma\geq
0$. For $\sigma>0$ we compute the corresponding Euler-Lagrange
equations in Lemma \ref{lemma:EL-eqn} below. By assuming further that
$\sigma> 1/8$, one can check the validity of the a-priori bounds of
Proposition  \ref{lemma:a-priori-bds}. These are instrumental for
passing to the limit for $\eps \to 0$. Our main result is the
convergence of minimizers up to subsequences to Leray-Hopf
solutions. 
 
\begin{theorem}[Variational approach to Navier-Stokes]\label{theo:main}
Let $\sigma > 1/8$ and $u^{\eps} \in U^{\eps}$ be
a minimizer of $I^{\eps}$. Then there exists a subsequence (not relabeled) such that 
\begin{align*} 
  u^{\eps} 
  \weakly u &\quad\mbox{in } L^2(0,\infty;V), \qquad 
  \partial_t u^{\eps} 
  \weakly \partial_t u \quad\mbox{in } L^{2}(0,\infty;V_s'),  
\end{align*} 
$s > 5/2$, for some $u \in L^2(0,\infty;V) \cap L^{\infty}(0,
\infty;H)$ with $u(0) = u_0$ where $u$ is a   Leray-Hopf
solution of the Navier-Stokes system \eqref{eq:NSE-with-p}. 
 Moreover, $u$ satisfies the energy inequality 
\begin{align}
   \| u(T) \|^2+ 2 \nu \int_0^T \| \nabla u(t) \|^2\, dt 
  \le   \| u_0 \|^2\label{eq:ei} 
  \quad\mbox{for a.e. } T > 0.  
\end{align}
\end{theorem}

The proof of the statement relies on a-priori estimates on the
Euler-Lagrange equations of $I^\epsi$. After deriving the Euler-Lagrange
equations  
in Section \ref{sec:EL}, we prove the a-priori bounds  in Section
\ref{sec:bounds}. Eventually, the proof of Theorem \ref{theo:main} is
in Section
\ref{sec:proof}. An interesting open problem raised by our analysis is, if convergence to solutions of the Navier-Stokes system can also be obtained in the unstabilized case $\sigma = 0$. 

Let us now comment on some possible variants of the statement. First
of all, the convergence holds more generally
for critical points $\ue$ of $I^\eps$, as long as $I^{\eps}(u^{\eps}) \le
C \eps^{-1}$. In case $u_0 \in V$ with $u_0\cdot \nabla u_0\in H$ we
may choose $u_0^{\eps} = u_0$ and no approximation as in
\eqref{eq:u-0-bd} is needed. For general $u_0$ and $\sigma >0$ 
this is not possible since $u \in 
U^{\eps}$ implies that $u \in L^2(0,1;V)$ and $\partial_t u \in
L^2(0,1;H)$ and hence $u(0) \in H^{1/2}(\Omega;\R^3)$. The
occurrence of external forces
can be considered as well.

 A caveat on notation: In the following, we use the same symbol $C$ for a positive
constant, possibly depending on $\sigma$, $\nu$, and $C_0$ but
independent of $\epsi$ and $T$. We warn the reader that the value of
$C$ may change from line to line. 

%--------------------------------------------------------------------------
%--------------------------------------------------------------------------
\section{Existence of  minimizers} \label{sec:min}
%--------------------------------------------------------------------------

We start by checking that $I^\eps$ admits minimizers on $U^\eps$
for all $\sigma \geq 0$ and $\eps > 0$, which are assumed to be fixed
throughout this section. 
\begin{proposition}[Existence of minimizers]\label{prop:min-ex} 
  Let $\sigma \geq 0$. Then, there exists a minimizer $u^{\eps}$ of $I^{\eps}$ in $U^{\eps}$ and
$I(u^{\eps}) \le C \eps^{-1}$.
\end{proposition}

\begin{proof} 
The set  $U^\eps$ in not empty,
as the constant-in-time trajectory $u \equiv u_0^{\eps}$ belongs to
it, see assumption \eqref{eq:u-0-bd}. In this case 
one has that 
\begin{align*}
  I^{\eps}(u) 
  &= \int_0^{\infty} \e^{-t/\eps}\Big\{\frac{1+ \sigma}{2} \| u_0^{\eps} \cdot \nabla u_0^{\eps} \|^2
     + \frac{\nu}{2\eps} \| \nabla u_0^{\eps} \|^2
  \Big\} \, dt \\ 
  &\le \frac{(1+ \sigma) \eps}{2} \| u_0^{\eps} \cdot \nabla u_0^{\eps} \|^2
     + \frac{\nu}{2} \| \nabla u_0^{\eps} \|^2
   \le C \eps^{-1}.
\end{align*} 
In particular, $\inf I^\epsi \le C \eps^{-1}$. Assume now $u_k \in
 U^\eps$ to be a minimizing sequence for $I^\eps$. By letting $v_k = \partial_t u_k + u_k \cdot
 \nabla u_k $ we have that 
$  \e^{-t/2\eps}v_k $ are bounded in $L^2(0,
\infty; L^2(\Omega; \R^3 ))$ and $\e^{-t/2\eps} u_k $ are bounded in 
 $L^2(0,
\infty; V)$.  One can hence extract not relabeled subsequences such
that $\e^{-t/2\eps} v_k \weakly \e^{-t/2\eps} v$ in $L^2(0,
\infty; L^2(\Omega;\R^3))$ and $\e^{-t/2\eps} u_k \weakly \e^{-t/2\eps} u$ in $L^2(0,
\infty; V)$. 

Fix any $T>0$ and note that $u_k$ are bounded in
$L^2(0,T;L^6(\Omega;\R^3))$ as $V \subset L^6(\Omega;\R^3)$. This
entails that $u_k \cdot
\nabla u_k$ are 
bounded in $L^1(0,T;L^{3/2}(\Omega;\R^3))$. As $v_k$ are bounded in
$L^2(0,T;L^2(\Omega;\R^3))$, we conclude that $\partial_t u_k$ are
bounded in $L^1(0,T;L^{3/2}(\Omega;\R^3))$ as well. Hence, $u_k$ are
bounded in $C([0,T];L^{3/2}(\Omega;\R^3))$ and, by
interpolation, in $L^4(0,T;L^{12/5}(\Omega;\R^3))$. Eventually, both $u_k \cdot
\nabla u_k$ and $\partial_t u_k$ are 
bounded in $L^{4/3}(0,T;L^{12/11}(\Omega;\R^3))$. By possibly
extracting again we find a
  not relabeled subsequence such that 
$$ \partial_t u_k \weakly \partial_t u \ \ \text{and} \
\ u_k \cdot \nabla u_k \weakly u\cdot \nabla u \ \ \text{in} \ \
L^{4/3}(0,T;L^{12/11}(\Omega;\R^3))$$
where for the last convergence we also used that 
$u_k \to u$ in $L^2(0,T;L^2(\Omega;\R^3))$ strongly due to the Aubin-Lions lemma 
and thus $u_k \cdot \nabla u_k \weakly u\cdot \nabla u$ in 
$L^{1}(0,T;L^{1}(\Omega;\R^3))$. 
This proves that 
$$v_k \weakly v=\partial_t u + u
\cdot \nabla u \ \ \text{in} \ \  L^2(0,T;L^2(\Omega;\R^3)).$$ 
We now use convexity in order to check that 
\begin{align*}
 &\int_0^T \e^{-t/\eps} \Big\{ \frac12 \| \partial_t u + u
\cdot \nabla u \|^2 + \frac{\sigma}{2}\|u \cdot
  \nabla u\|^2 + \frac{\nu}{2\eps}\| \nabla u\|^2 \Big\} dt \nonumber\\
&\quad\leq
  \liminf_{k \to \infty} \int_0^T \e^{-t/\eps} \Big\{ \frac12 \| \partial_t u_k + u_k
\cdot \nabla u_k \|^2 + \frac{\sigma}{2}\|u_k \cdot
  \nabla u_k\|^2 + \frac{\nu}{2\eps}\| \nabla u_k\|^2 \Big\} dt \nonumber\\
&\quad\leq
  \liminf_{k \to \infty}  I^\eps(u_k) = \inf I^\eps.
\end{align*}
 It now suffices to take the limit as $T\to \infty$ on the
left-hand side above to conclude that $I^\eps(u) = \min I^\eps$.
\end{proof} 

For later use we remark that for $u\in U^{\eps}$ with 
$ I^\eps  (u) < \infty$ it holds 
\begin{align}\label{eq:u-eps-exp-bd} 
  \begin{split}
    &\e^{-t/2\eps} (\partial_t u + u \cdot\nabla u), \ \sigma\,
    \e^{-t/2\eps} u \cdot \nabla u
    \in L^2(0,\infty;L^2(\Omega; \R^3 )) \\
    &\quad \text{and} \ \ \e^{-t/2\eps} \nabla u\in
    L^2(0,\infty;L^2(\Omega; \R^{3\times 3} )).
  \end{split}
\end{align}

%--------------------------------------------------------------------------
%--------------------------------------------------------------------------
\section{Euler-Lagrange equations} \label{sec:EL}
%--------------------------------------------------------------------------

Having established the existence of minimizers, we will now derive the
corresponding Euler-Lagrange equations. The stabilization parameter $\sigma>0$ is
kept fixed throughout the section.

\begin{lemma}[Euler-Lagrange equations]\label{lemma:EL-eqn}
Let $\sigma>0$ and $u^{\eps}$ be a minimizer of $I^{\eps}$. For any $\varphi \in L^2_{\rm loc}(0,\infty;V)$ with $\partial_t \varphi \in L^2_{\rm loc}(0,\infty;H)$, $\varphi(0) = 0$ and such that 
\begin{align*}
 &\e^{-t/2\eps} \partial_t \varphi, 
   \e^{-t/2\eps} \varphi \cdot \nabla \varphi,~ 
   \e^{-t/2\eps} \varphi \cdot \nabla u^{\eps},~ 
   \e^{-t/2\eps} u^{\eps} \cdot \nabla \varphi \in
   L^2(0,\infty;L^2(\Omega;  \R^3 )) \\
&~~\text{and}  \ \ ~ 
   \e^{-t/2\eps} \nabla \varphi \in L^2(0,\infty;L^2(\Omega;  \R^{3\times 3} ))
\end{align*}  we have   that 
\begin{align}\label{eq:EL-weak-eps}
\begin{split}
  &\int_0^{\infty} \e^{-t/\eps} \int_{\Omega} 
        \big( \partial_t u^{\eps} + u^{\eps} \cdot \nabla u^{\eps}  \big) 
       \cdot \big( \partial_t \varphi + u^{\eps} \cdot \nabla \varphi + \varphi \cdot \nabla u^{\eps} \big) \\ 
  &\qquad \qquad \qquad
     + \sigma (u^{\eps} \cdot \nabla u^{\eps}) \cdot \big( u^{\eps} \cdot \nabla \varphi + \varphi \cdot \nabla u^{\eps} \big) 
     + \frac{\nu}{\eps} \nabla u^{\eps} \cdot \nabla \varphi 
     \, dx \, dt = 0. 
\end{split}
\end{align}
In particular, any $\varphi \in C^{\infty}_{\rm c}(0,\infty; V_s)$, $s > 5/2$, satisfies 
\begin{align}\label{eq:EL-weak} 
\begin{split}
  &\int_0^{\infty} \int_{\Omega}  ( \partial_t u^{\eps} + u^{\eps} \cdot \nabla u^{\eps} ) \cdot 
    \big[ \varphi + \eps \partial_t \varphi 
    + \eps \varphi \cdot \nabla u^{\eps} + \eps u^{\eps} \cdot \nabla \varphi \big]  \\ 
 &\qquad + {}\sigma \eps (u^{\eps} \cdot \nabla u^{\eps}) \cdot \big[ \varphi \cdot \nabla u^{\eps} + u^{\eps} \cdot \nabla \varphi \big] 
  + \nu \nabla u^{\eps} : \nabla \varphi  \, dx \, dt 
  = 0. 
\end{split}
\end{align}
\end{lemma}

\begin{proof}
The assumptions guarantee that $u^{\eps} + s \varphi \in U^{\eps}$ and 
\begin{align*}
 I^{\eps}(u^{\eps} + s \varphi) &= \int_0^{\infty} \e^{-t/\eps}\bigg\{
   \frac{1}{2} \| \partial_t (u^{\eps} + s \varphi) 
       + (u^{\eps} + s \varphi) \cdot \nabla (u^{\eps} + s \varphi) \|^2\\ 
 &  \qquad+ \frac{\sigma}{2} \| (u^{\eps} + s \varphi) 
         \cdot \nabla (u^{\eps} + s \varphi) \|^2
    + \frac{\nu}{2\eps} \| \nabla (u^{\eps} + s \varphi) \|^2 
 \bigg\} \, dt 
\end{align*}
is finite for any $s \in \R$  and is minimized at  $s =
0$.  Indeed, by letting  $\zeta_1 
= \partial_t \varphi + u^{\eps} \cdot \nabla \varphi + \varphi \cdot \nabla u^{\eps}$, 
$\zeta_2 = \varphi \cdot \nabla \varphi$, and $\xi_1 = u^{\eps} \cdot \nabla \varphi 
+ \varphi \cdot \nabla u^{\eps}$, $\xi_2 = \varphi \cdot \nabla
\varphi$  one can write 
\begin{align*}
 I^{\eps}(u^{\eps} + s \varphi)  &=\int_0^{\infty} \e^{-t/\eps}\bigg\{\frac{1}{2} \| \partial_t u^{\eps} 
      + u^{\eps} \cdot \nabla u^{\eps} 
      + s \zeta_1 + s^2 \zeta_2 \|^2\\ 
 &\qquad   + \frac{\sigma}{2} \| u^{\eps} \cdot \nabla u^{\eps} 
      + s \xi_1 + s^2 \xi_2 \|^2
    + \frac{\nu}{2\eps} \| \nabla u^{\eps} + s \nabla \varphi \|^2
 \bigg\} \, dt 
\end{align*}
By assumption we have $\e^{-t/2\eps} \zeta_1, \e^{-t/2\eps} \zeta_2, \e^{-t/2\eps} \xi_1, 
\e^{-t/2\eps} \xi_2 \in L^2(0,\infty;L^2(\Omega; \R^3))$. In combination with \eqref{eq:u-eps-exp-bd} this
shows that $I^{\eps}(u^{\eps} + s \varphi)$ is differentiable with  
\begin{align*}
 0 
 &= \frac{d}{ds}\Big|_{s = 0} I^{\eps}(u^{\eps} + s \varphi) \\
 &= \int_0^{\infty} \e^{-t/\eps} \int_{\Omega} 
       \big( \partial_t u^{\eps} + u^{\eps} \cdot \nabla u^{\eps} \big) \cdot \zeta_1 
      + \sigma (u^{\eps} \cdot \nabla u^{\eps}) \cdot \xi_1  
      + \frac{\nu}{\eps} \nabla u^{\eps} : \nabla \varphi 
      \, dx \, dt 
\end{align*}
as claimed. 

Consider now $\varphi \in C^{\infty}_{\rm c}(0,\infty,V_s)$, $s > 5/2$, 
and let $\psi = \e^{t/\eps} \varphi$ so that $\partial_t \psi =
\e^{t/\eps} \partial_t \varphi +  \eps^{-1} 
{\e^{t/\eps}}\varphi\in L^2(0,\infty;L^2(\Omega;  \R^3 ))$. 
Since $V_s \hookrightarrow
W^{1,\infty}(\Omega; \R^3 )$ and $\varphi$ has compact support
in $t$ we infer from  \eqref{eq:u-eps-exp-bd} that $\psi$ satisfies
the assumptions of the first part of the Lemma. Equation  \eqref{eq:EL-weak} is now a direct consequence of \eqref{eq:EL-weak-eps} applied to $\psi$. 
\end{proof}

We may  equivalently write equality  \eqref{eq:EL-weak} as 
\begin{align}\label{eq:EL-nach-phi}
\begin{split}
  &\int_0^\infty \int_{\Omega} 
    \eps  \big[ \partial_t u^{\eps} + u^{\eps} \cdot \nabla u^{\eps} \big] \cdot \partial_t \varphi 
    +  \big[ \partial_t u^{\eps} + u^{\eps} \cdot \nabla u^{\eps} \big] \cdot \varphi \\ 
    &\qquad\qquad   + \eps (\nabla u^{\eps})^\top \big[  \partial_t u^{\eps} + (1+ \sigma) u^{\eps} \cdot \nabla u^{\eps} \big] \cdot \varphi \\ 
  &\qquad\qquad  + \big[ \eps \big(  \partial_t u^{\eps} + (1 + \sigma) u^{\eps} \cdot \nabla u^{\eps} \big) \otimes u^{\eps} + \nu \nabla u^{\eps} \big] : \nabla \varphi  \, dx \, dt 
   = 0
\end{split}
\end{align}
 which corresponds indeed to the weak formulation of \eqref{EL}. By
using the  quadratic function $B$ introduced in \eqref{eq:B-def} we furthermore set 
\begin{align}\label{eq:v-def}
 v^{\eps} 
   =   \partial_t u^{\eps} +  B(u^{\eps}) =    \partial_t u^{\eps}+  
   u^{\eps}  \cdot \nabla u^{\eps} 
\end{align}
 where the latter equality follows from the fact that $u^{\eps}  \cdot
\nabla u^{\eps}  \in L^2(\Omega; \R^3 )$. (More precisely, $B(u^{\eps}) \in H \subset V'$ is the $L^2$-orthogonal projection of $u^{\eps}  \cdot \nabla u^{\eps}$ onto $H$.) Since $u^{\eps}, \partial_t u^{\eps},
u^{\eps} \cdot \nabla u^{\eps} \in L^2_{\rm
  loc}(0,\infty; L^2(\Omega; \R^3 ))$,  $\nabla
u^{\eps}\in L^2_{\rm
  loc}(0,\infty; L^2(\Omega; \R^{3\times 3} ))$,  and
$V_s \hookrightarrow W^{1,\infty}(\Omega; \R^3 )$ 
for  $s > 5/2$, we have that $f^{\eps}$ and $g^{\eps}$, defined by 
\begin{align}
  \langle f^{\eps}(t), \psi \rangle 
  &= \int_{\Omega} \eps (\nabla u^{\eps})^\top \big[  \partial_t u^{\eps}
  + (1  + \sigma) u^{\eps} \cdot \nabla u^{\eps} \big] \cdot \psi \, dx, \label{eq:f-def} \\ 
   \langle g^{\eps}(t), \psi \rangle 
  &= \int_{\Omega} \eps \big[  \partial_t u^{\eps} + (1  + \sigma) u^{\eps} \cdot \nabla u^{\eps} \big] \otimes u^{\eps} : \nabla \psi \, dx \label{eq:g-def}
\end{align}
for a.e.\ $t$ for all $\psi \in V_s$ are elements of $L^1_{\rm loc}(0, \infty; V_s')$ with 
\begin{align}
  \| f^{\eps} \|_{L^1(0, T; V_s')} 
  &\le C \eps \| \nabla u^{\eps} \|_{L^2(0,T;L^2(\Omega;\R^{3\times 3}))}
  \|  \partial_t u^{\eps} + (1   + \sigma) u^{\eps} \cdot \nabla u^{\eps} \|_{L^2(0,T;L^2(\Omega;\R^{n}))}, \label{eq:f-bd} \\ 
  \| g^{\eps} \|_{L^1(0, T; V_s')} 
  &\le C \eps \|  \partial_t u^{\eps} + (1    + \sigma) u^{\eps} \cdot \nabla u^{\eps} \|_{L^2(0,T;L^2(\Omega;\R^{n}))}\| u^{\eps} \|_{L^2(0,T;L^2(\Omega;\R^{n}))}\label{eq:g-bd} 
\end{align}
for any $T > 0$. 
By recalling the definition \eqref{eq:A-def} of the Stokes operator $A
: V \to V' \hookrightarrow V_s'$, we  obtain from Lemma
\ref{lemma:EL-eqn}  the following. 

\begin{corollary}[Euler-Lagrange equations, strong form]\label{cor:EL-strong}
Let $u^{\eps}$ be a minimizer of $I^{\eps}$  on $U^\eps$  and
$v^{\eps}, f^{\eps}, g^{\eps}$  be  defined in \eqref{eq:v-def}, \eqref{eq:f-def}, \eqref{eq:g-def}, respectively. Then 
\begin{align}\label{eq:euler-v-distr-weak}
  \int_0^{\infty} \big( \eps v^{\eps}, \partial_t \varphi \big)_H 
     + \langle v^{\eps} + f^{\eps} + g^{\eps}  +  \nu A u^{\eps}, \varphi \rangle \, dt
  = 0 
\end{align}
for all $\varphi\in C^{\infty}_{\rm c}(0,\infty;V_s)$, $s > 5/2$. In fact, $u^{\eps}$ is a solution to the elliptic problem
\begin{align}\label{eq:euler-v-distr}
  -\eps \partial_{t} v^{\eps}
  + v^{\eps}
  + f^{\eps} + g^{\eps}  +   \nu A u^{\eps}
  = 0
\end{align}
in $L^1_{\rm loc}(0, \infty; V_s')$. 
\end{corollary} 

\begin{proof}
First note that  relation  \eqref{eq:euler-v-distr-weak} is an
immediate consequence of  the former 
\eqref{eq:EL-nach-phi}. Hence,  equation 
\eqref{eq:euler-v-distr} holds in the sense of distributions
$\mathcal{D}'(0,\infty;V_s')$ on $(0,\infty)$ with values in
$V_s'$. In particular, as $v^{\eps}, f^{\eps}, g^{\eps}, A u^{\eps}
\in L^1_{\rm loc}(0, \infty; V_s')$ we  have that $\partial_{t}
v^{\eps} \in L^1_{\rm loc}(0, \infty; V_s')$ as well.  
\end{proof}

%--------------------------------------------------------------------------
%--------------------------------------------------------------------------
\section{A-priori estimates}\label{sec:bounds}
%--------------------------------------------------------------------------

In order to pass to the limit $\eps\to 0$  in the Euler-Lagrange equations and prove Theorem
\ref{theo:main} we now provide estimates on the minimizers
$u^{\eps}$. This requires   $\sigma
>1/8$, which is assumed throughout this section. 
 At first, moving from Lemma
\ref{lemma:EL-eqn} we prove in Proposition \ref{lemma:a-priori-bds} an
energy estimate which formally consists in testing the Euler-Lagrange
equations by $\ue$. Then, Proposition \ref{prop:dualbounds} exploits
the strong form of the Euler-Lagrange equations from Corollary
\ref{cor:EL-strong} in order to deduce bounds in dual spaces. Their derivation 
with the help of the convolution kernel $K$ (see \eqref{eq:K-kernel}) is inspired by a 
similar reasoning in \cite{Lions:65}.

\begin{proposition}[Energy estimate]\label{lemma:a-priori-bds}
 Let $\sigma > 1/8$. All minimizers   $u^{\eps}\in U^\eps$
 of $I^{\eps}$  satisfy 
$ \partial_t u^{\eps} ,
\,  
   u^{\eps} \cdot \nabla u^{\eps}  \in L^2(0,\infty;L^2(\Omega; \R^3))  $,
$ u^{\eps} \in L^2(0,\infty;H^1(\Omega;  \R^3 )) \cap
L^{\infty}(0,\infty;L^2(\Omega;  \R^3 )) $,  and 
$$ \| u^{\eps} \|_{L^2(0,\infty;H^1(\Omega;  \R^3 ))} 
   + \| u^{\eps} \|_{L^{\infty}(0,\infty;L^2(\Omega;  \R^3 ))} 
   \le C $$
as well as 
$$ \sqrt{\eps} \| \partial_t u^{\eps} \|_{L^2(0,\infty;L^2(\Omega;  \R^3 ))} 
   + \sqrt{\eps} \| u^{\eps} \cdot \nabla u^{\eps} \|_{L^2(0,\infty;L^2(\Omega;  \R^3 ))}
   \le C. $$ 
More explicitly, the following inequality holds 
\begin{align}\label{eq:unif-bd}
   \|  u^{\eps}(T) \|^2  + 2 \nu \int_0^T (1 - \e^{-t/\eps})
  \| \nabla u^{\eps}(t) \|^2 \, dt
  \le  \|  u_0^{\eps} \|^2
\end{align} 
for all $T > 0$. 
\end{proposition}

\begin{proof}
Let us start by estimating  $u^{\eps}$ on  the   initial
short-time interval  $(0,\eps)$. 
By letting   $c = {\min\{1,\sigma,\nu\}}/({2 \e})>0$ we have 
\begin{align*}
  &\frac{c}{\eps} \int_0^{\eps} 
     \eps \| \partial_t u^{\eps} + u^{\eps} \cdot \nabla u^{\eps}  \|^2
   + \eps \| u^{\eps} \cdot \nabla u^{\eps} \|^2+ \| \nabla u^{\eps} \|^2\, dt \\ 
  &~~\le \int_0^{\infty} \e^{-t/\eps} \Big\{
      \frac{1}{2} \| \partial_t u^{\eps} + u^{\eps} \cdot \nabla u^{\eps} \|^2
      + \frac{\sigma}{2} \| u^{\eps} \cdot \nabla u^{\eps} \|^2
      + \frac{\nu}{2\eps} \| \nabla u^{\eps} \|^2\Big\} \, dt \\ 
  &~~ =  I(u^{\eps}) 
   \le C \eps^{-1}  
\end{align*}
by Proposition \ref{prop:min-ex}, and thus 
\begin{align}\label{eq:small-time-bd}
  \int_0^{\eps} \eps \| \partial_t u^{\eps} \|^2
   + \eps \| u^{\eps} \cdot \nabla u^{\eps} \|^2
   + \| \nabla u^{\eps} \|^2\, dt 
  \le C. 
\end{align}

 In order to deduce  an estimate on $[\eps, T]$,  for   $T > 0$, we  let $\eta \in W^{1,\infty}(0, \infty)$ be
defined  as  
$$ \eta(t) 
   = \begin{cases} 
        \e^{t/\eps} - 1 &\mbox{for } t \le T, \\ 
        (\e^{T/\eps} - 1) &\mbox{for } t \ge T. 
     \end{cases} $$ 
 and apply Lemma \ref{lemma:EL-eqn} to $\varphi = \eta \ue$. We  obtain  
\begin{align*}
  0 
  &= \int_0^{\infty} \e^{-t/\eps} \int_{\Omega} \eta(t) \Big\{ ( \partial_t u^{\eps} + u^{\eps} \cdot \nabla u^{\eps} ) \cdot 
    \big[ \partial_t u^{\eps} + 2 u^{\eps} \cdot \nabla u^{\eps} \big] \\ 
  &\qquad\qquad\qquad 
    + 2 \sigma |u^{\eps} \cdot \nabla u^{\eps}|^2 
    + \frac{\nu}{\eps} |\nabla u^{\eps}|^2 \Big\} 
 + \eta'(t)  ( \partial_t u^{\eps} + u^{\eps} \cdot \nabla u^{\eps} ) \cdot u^{\eps} \, dx \, dt \\ 
%-------------------
  &= \int_0^{\infty} \e^{-t/\eps} \eta(t) \int_{\Omega} 
      |\partial_t u^{\eps}|^2 
     + 3  \partial_t u^{\eps} \cdot (u^{\eps} \cdot \nabla u^{\eps} ) 
     + 2 (1 + \sigma) |u^{\eps} \cdot \nabla u^{\eps}|^2 
   + \frac{\nu}{\eps} |\nabla u^{\eps}|^2  \, dx \, dt \\ 
  &\qquad + \int_0^T \frac{1}{\eps} \int_{\Omega}  \partial_t u^{\eps} \cdot u^{\eps} \, dx \, dt, 
\end{align*}
where we have  also  used the fact  that $\int_{\Omega} (u^{\eps} \cdot \nabla u^{\eps}) \cdot u^{\eps} \, dx  = 0$. It follows that 
\begin{align*}
\begin{split}
  0 
  &= \int_0^{\infty} \e^{-t/\eps} \eta(t) \int_{\Omega} 
     \eps  |\partial_t u^{\eps} + \tfrac{3}{2} u^{\eps} \cdot \nabla u^{\eps}|^2 
+ \eps (2 \sigma - \tfrac{1}{4} ) |u^{\eps} \cdot \nabla u^{\eps}|^2 + \nu |\nabla u^{\eps}|^2 \, dx \ dt  \\ 
  &\qquad + \frac{1}{2} \int_{\Omega} |u^{\eps}(T)|^2dx - 
  \frac{1}{2} \int_{\Omega}   |u^{\eps}(0)|^2 \, dx.   
\end{split}
\end{align*}

Since $\eta(t) \ge 0$ for all $t$ and $t \mapsto \e^{-t/\eps} \eta(t)
= 1 - \e^{-t/\eps}$ increases on $[0, T]$,  by defining  $c =
\min\{1, 2\sigma - 1/4, \nu\}(1 - \e^{-1}) >0 $ we get 
\begin{align}\label{eq:large-time-bd}
\begin{split}
  &c \int_{\eps}^T \int_{\Omega} 
    \eps |\partial_t u^{\eps} + \tfrac{3}{2} u^{\eps} \cdot \nabla u^{\eps}|^2 
    + \eps |u^{\eps} \cdot \nabla u^{\eps}|^2
    + |\nabla u^{\eps}|^2 \, dx \ dt \\ 
  & \le \int_0^T \e^{-t/\eps} \eta(t) \int_{\Omega} 
     \eps  |\partial_t u^{\eps} + \tfrac{3}{2} u^{\eps} \cdot
     \nabla u^{\eps}|^2 
+ \eps (2 \sigma - \tfrac{1}{4} ) |u^{\eps} \cdot \nabla u^{\eps}|^2 + \nu |\nabla u^{\eps}|^2 \, dx \ dt  \\ 
  &\le \frac{1}{2} \| u_0^{\eps} \|^2 -
  \frac{1}{2} \| \ue(T) \|^2. 
\end{split}
\end{align}
Combining \eqref{eq:small-time-bd}, \eqref{eq:large-time-bd}, and \eqref{eq:u-0-bd} we see that 
$$  \| \ue(T) \|^2+  \int_0^T 
    \eps \| \partial_t u^{\eps} \|^2 
    + \eps \| u^{\eps} \cdot \nabla u^{\eps} \|^2 
    + \| \nabla u^{\eps} \|^2 \, dt \le C. $$  
As $T > 0$  is  arbitrary, this proves  the
statement. 
\end{proof}

\begin{proposition}[Dual estimate]\label{prop:dualbounds}
Let $\sigma >1/8$ and $s > 5/2$. All minimizers $\ue \in U^\epsi$
of $I^\eps$ satisfy 
\begin{align*}
  \|\partial_t u^{\eps}\|_{L^2(0,\infty;V_s')} 
  + \|B(u^{\eps})\|_{L^2(0,\infty;V_s')}  \le C. 
\end{align*}
\end{proposition}

\begin{proof}
Let  $0 \le \tau \le T$. By  Proposition 
\ref{lemma:a-priori-bds} and equations \eqref{eq:f-bd},
\eqref{eq:g-bd},  and  \eqref{eq:Au-bd} we have the estimates 
\begin{align}\label{eq:h-bd}
  \| f^{\eps} + g^{\eps} \|_{L^1(0, \infty;V_s')} 
  \le C \sqrt{\eps}  
  \qquad \mbox{and } \qquad 
  \| A u^{\eps} \|_{L^2(0, \infty;V_s')} 
  \le C. 
\end{align}
Multiplying  the Euler-Lagrange equation 
\eqref{eq:euler-v-distr} by $ \epsi^{-1} \e^{-t/\eps}$, integrating over $[\tau, T]$, and multiplying by $\e^{\tau/\eps}$, we find that $v^{\eps}$ as defined in \eqref{eq:v-def} satisfies 
\begin{align*} 
  v(\tau)
  = \e^{(\tau-T)/\eps} v(T) - \int_{\tau}^T\frac{\e^{(\tau-t)/\eps}}{\eps} \big( f^{\eps}(t) + g^{\eps}(t)  + \nu Au^{\eps}(t) \big) \, dt.   
\end{align*}
As $v^{\eps} \in L^2(0,\infty;H)$ by  Proposition  \ref{lemma:a-priori-bds} we may send $T \to \infty$ along a sequence with $v^{\eps}(T) \to 0$ in $H$ ($\hookrightarrow V_s'$) and see that 
\begin{align*} 
  v^{\eps}(\tau)
  = - \int_{\tau}^{\infty}\frac{\e^{(\tau-t)/\eps}}{\eps} \big(
  f^{\eps}(t) + g^{\eps}(t)  + \nu Au^{\eps}(t) \big) \, dt.
\end{align*}

Extending $f^{\eps}, g^{\eps}$ and $A u^{\eps}$ by $0$ on
$(-\infty,0)$ to all of $\R$, this can be written as
$$ v^{\eps} 
   = K * (f^{\eps} + g^{\eps}  + \nu A u^{\eps}) $$ 
on  $[0, \infty)$  where the kernel $K$ is given by 
\begin{align}\label{eq:K-kernel} 
  K(t) = 
  \left\{
  \begin{array}{ll}
  {\eps^{-1}}{\e^{t/\eps}} \quad&\text{for} \  \
                                                        t\leq 0,\\
0 \quad&\text{for} \  \
                                                        t> 0.
 \end{array}
\right.
\end{align} 
This  entails  that 
\begin{align}\label{eq:v-eps-L2-bd}
  \| v^{\eps} \|_{L^2(0,\infty;V_s')}
  \le C  
\end{align}
since by Young's inequality and \eqref{eq:h-bd} we can estimate 
\begin{align*}
  \| K * (f^{\eps} + g^{\eps}) \|_{L^2(\R;V_s')}
  &\le \| K \|_{L^2(\R)} \| f^{\eps} + g^{\eps} \|_{L^1(\R;V_s')} 
  \le C 
\end{align*}
and 
\begin{align*}
  \| \nu K * A u^{\eps} \|_{L^2(\R;V_s')}
  &\le \| K \|_{L^1(\R)} \| \nu A u^{\eps} \|_{L^2(\R;V_s')} 
  \le C, 
\end{align*}
where we have used that $\| K \|_{L^1(\R)} = 1$ and $\| K \|_{L^2(\R)} = {1}/{\sqrt{2\eps}}$.  
As by  Proposition  \ref{lemma:a-priori-bds} and bound
\eqref{piu} 
one also has that  
\begin{align*}
  \| B(u^{\eps}) \|_{L^2(0,\infty;V_s')} 
  &\le C \| u^{\eps} \cdot \nabla u^{\eps} \|_{L^2(0,\infty;L^1(\Omega; \R^3 ))} \\ 
  &\le C \| u^{\eps} \|_{L^{\infty}(0,\infty;L^2(\Omega; \R^3 ))} \| \nabla u^{\eps} \|_{L^2(0,\infty;L^2(\Omega; \R^3 ))} 
  \le C, 
\end{align*}
the claim follows from this inequality and  \eqref{eq:v-eps-L2-bd}.  
\end{proof}

%--------------------------------------------------------------------------
%--------------------------------------------------------------------------
\section{Convergence}\label{sec:proof}
%--------------------------------------------------------------------------

This section is eventually devoted to the proof of Theorem
\ref{theo:main}. Let
$u^{\eps}$ be a minimizer of $I^{\eps}$  on $U^\eps$  and
define $v^{\eps}$ as in \eqref{eq:v-def}. By  the a-priori
estimates of Propositions  \ref{lemma:a-priori-bds} and
\ref{prop:dualbounds} (recall that we assume here $\sigma>1/8$) there exists $u \in L^2(0,\infty;V) \cap L^{\infty}(0, \infty; H)$ with $\partial_t u \in L^2(0,\infty;V_s')$, $s > 5/2$, such that, for a subsequence (not relabeled), 
\begin{align}\label{eq:u-eps-conv} 
  u^{\eps} 
  \weakly u &\quad\mbox{in } L^2(0,\infty;V), \qquad 
  u^{\eps} 
  \weaklystar u \quad\mbox{in } L^{\infty}(0,\infty;H) 
\end{align} 
as well as 
\begin{align}\label{eq:partial-t-conv}
  \partial_t u^{\eps} 
  \weakly \partial_t u \quad\mbox{in } L^{2}(0,\infty;V_s'). 
\end{align} 
This, in particular, entails  that 
\begin{align}\label{eq:Au-conv}
  A u^{\eps} \weakly A u 
  \quad\mbox{in } L^2(0,\infty;V_s') 
\end{align}
 as well, since $A$ is bounded, see 
\eqref{eq:Au-bd}. Moreover, \eqref{eq:v-def}, \eqref{eq:f-bd}, and
\eqref{eq:g-bd} in combination with  Proposition  \ref{lemma:a-priori-bds} yield 
\begin{align} 
  \eps v^{\eps}  
  &\to 0 
  \quad\mbox{in } L^2(0,\infty;L^2(\Omega; \R^3)) \label{eq:eps-v-conv} \qquad \mbox{and} \\  
  f^{\eps} + g^{\eps} 
  &\to 0 \quad \mbox{in } L^1(0,\infty;V_s'). \label{eq:fg-conv}
\end{align}
By the Aubin-Lions Lemma we obtain $u^{\eps} \to u$ strongly in $L^2_{\rm loc}(0, \infty; H)$ so that \eqref{eq:u-eps-conv}, Proposition \ref{prop:dualbounds}, and $V_s \hookrightarrow L^{\infty}(\Omega)$ yield  
\begin{align}\label{eq:nonlin-lim-L1} 
  B(u^{\eps}) 
  \weakly B(u) 
  \quad\mbox{in } L^2(0,\infty;V_s'). 
\end{align}
From \eqref{eq:partial-t-conv} and \eqref{eq:nonlin-lim-L1} we thus get 
\begin{align}\label{eq:v-conv} 
  v^{\eps}   
  \weakly   \partial_t u +  B(u) 
  \quad\mbox{in } L^2(0,\infty;V_s').   
\end{align}

 Recall now  from Corollary \ref{cor:EL-strong} that
$u^{\eps}$ is a solution to \eqref{eq:euler-v-distr-weak} for any
$\varphi\in C^{\infty}_{\rm c}((0,\infty);V_s)$.  By using  \eqref{eq:eps-v-conv}, \eqref{eq:v-conv}, \eqref{eq:fg-conv}, and \eqref{eq:Au-conv} we can now pass to the limit $\eps \to 0$ in \eqref{eq:euler-v-distr-weak} and get 
\begin{align*}
  \int_0^{\infty} \langle  \partial_t u +  B(u)  +  \nu A u, \varphi \rangle \, dt
  = 0 
\end{align*}
for any $\varphi\in C^{\infty}_c((0,\infty);V_s)$ and so 
\begin{align}\label{eq:NSE-strong}
  \partial_t u + \ B(u) + \nu A u 
  = 0 
\end{align}
in $L^2(0, \infty; V_s')$ and a.e.\ in time.   

Since $W^{1,1}(0,T,V_s') \hookrightarrow C([0,T], V_s')$ for any $T >
0$, we also have $u^{\eps} \weakly u$ in $C([0,T], V_s')$ and   $u$
satisfies the initial condition $u(0) = u_0$ because of
\eqref{eq:u-0-bd}. As $u \in L^2(0,\infty;V)$, the bounds
\eqref{eq:Au-bd} and \eqref{eq:B-bd} entail that $Au \in
L^2(0,\infty;V')$ and $B(u)\in L^1(0,\infty;V')$, respectively. By
comparison in \eqref{eq:NSE-strong} one has that $\partial_t u \in
L^1(0,\infty;V')$ as well so that the equation holds in $V'$ for a.a.
times and $u$ is indeed a Leray-Hopf
solution. 

 Eventually,  the energy inequality  \eqref{eq:ei}
 follows  by  lower
semicontinuity when passing to the limit $\eps \to 0$ in 
inequality 
\eqref{eq:unif-bd}  by using the convergence $\ue_0 \to u_0$ in
$H$ from  \eqref{eq:u-0-bd}, $u^{\eps}(T) \weakly u(T)$
in $L^2(\Omega; \R^3)$,  and  $(1 - \e^{-t/\eps})^{1/2} \nabla u^{\eps} \weakly
\nabla u$ in $L^2(0, T; L^2(\Omega; \R^{3\times 3} ))$.

%--------------------------------------------------------------------------
%--------------------------------------------------------------------------
\section*{Acknowledgements}
%--------------------------------------------------------------------------

U.S. is partly funded by the Vienna Science and Technology Fund (WWTF)
through Project MA14-009 and  by the Austrian Science Fund (FWF)
projects F\,65, P\,27052, and I\,2375.

%--------------------------------------------------------------------------
%--------------------------------------------------------------------------

\bibliographystyle{alpha}

\def\cprime{$'$} \def\cprime{$'$}

\end{document}